\documentclass[12pt]{amsart}

\usepackage{amsxtra,amssymb,amsthm,amsmath,amscd,listings}
\usepackage[hyphens]{url}
\usepackage[utf8]{inputenc}
\usepackage{eucal}
\usepackage{fullpage}
\usepackage{scrtime}
\usepackage{hyperref}

\renewcommand{\leq}{\leqslant}
\renewcommand{\geq}{\geqslant}

\numberwithin{equation}{section}

\newcommand{\Cc}{\mathbf{C}}

\newcommand{\Aa}{\mathbf{A}}

\newcommand{\Zz}{\mathbf{Z}}
\newcommand{\Pp}{\mathbf{P}}
\newcommand{\Rr}{\mathbf{R}}
\newcommand{\Gg}{\mathbf{G}}

\newcommand{\Qq}{\mathbf{Q}}
\newcommand{\Fp}{\mathbf{F}_p}

\newcommand{\proba}{\mathbf{P}}
\newcommand{\expect}{\mathbf{E}}
\newcommand{\frtr}[3]{t_{{{#1},{#2}}}({#3})}
\newcommand{\frfn}[2]{t_{{{#1},{#2}}}}

\newcommand{\tnorm}[2]{\|{#1}\|_{\mathrm{tr},{#2}}}

\newcommand{\ra}{\rightarrow}
\newcommand{\lra}{\longrightarrow}

\newcommand{\injecte}{\hookrightarrow}

\DeclareMathOperator{\rank}{rank}

\DeclareMathOperator{\Imag}{Im}
\DeclareMathOperator{\Reel}{Re}

\DeclareMathOperator{\frob}{\mathrm{Fr}}

\DeclareMathOperator{\Tr}{Tr}

\DeclareMathOperator{\swan}{Swan}

\DeclareMathOperator{\ft}{FT}
\DeclareMathOperator{\cond}{\mathbf{c}}
\DeclareMathOperator{\mext}{\mathsf{ME}}
\DeclareMathOperator{\mexti}{\mathbf{ME}}
\DeclareMathOperator{\lisse}{\mathsf{L}}
\DeclareMathOperator{\lissei}{\mathbf{L}}
\DeclareMathOperator{\sing}{Sing}

\newcommand{\eps}{\varepsilon}
\renewcommand{\rho}{\varrho}

\DeclareMathOperator{\GL}{GL}

\newcommand{\demi}{{\textstyle{\frac{1}{2}}}}

\newcommand{\sheaf}[1]{\mathcal{{#1}}}

\DeclareMathSymbol{\gena}{\mathord}{letters}{"3C}
\DeclareMathSymbol{\genb}{\mathord}{letters}{"3E}

\theoremstyle{plain}
\newtheorem{theorem}{Theorem}[section]
\newtheorem{lemma}[theorem]{Lemma}
\newtheorem{corollary}[theorem]{Corollary}

\newtheorem{proposition}[theorem]{Proposition}

\theoremstyle{remark}

\theoremstyle{definition}

\newtheorem{definition}[theorem]{Definition}

\newtheorem{remark}[theorem]{Remark}

\renewcommand{\geq}{\geqslant}
\renewcommand{\leq}{\leqslant}

\begin{document}

\title{Counting sheaves using spherical codes}

\author{\'Etienne Fouvry}
\address{Universit\'e Paris Sud, Laboratoire de Math\'ematique\\
  Campus d'Orsay\\ 91405 Orsay Cedex\\France}
\email{etienne.fouvry@math.u-psud.fr} 

\author{Emmanuel  Kowalski}
\address{ETH Z\"urich -- D-MATH\\
  R\"amistrasse 101\\
  8092 Z\"urich\\
  Switzerland} 
\email{kowalski@math.ethz.ch}

\author{Philippe Michel}
\address{EPFL/SB/IMB/TAN, Station 8, CH-1015 Lausanne, Switzerland }
\email{philippe.michel@epfl.ch}

\date{\today,\ \thistime}

\subjclass[2010]{11G20,11T23,94B60,94B65} 

\keywords{Lisse $\ell$-adic sheaves, trace functions, spherical codes,
  Riemann Hypothesis over finite fields}

\begin{abstract}
  Using the Riemann Hypothesis over finite fields and bounds for the
  size of spherical codes, we give explicit upper bounds, of
  polynomial size with respect to the size of the field, for the
  number of geometric isomorphism classes of geometrically irreducible
  $\ell$-adic middle-extension sheaves on a curve over a finite field
  which are pointwise pure of weight $0$ and have bounded ramification
  and rank. As an application, we show that ``random'' functions
  defined on a finite field can not usually be approximated by short
  linear combinations of trace functions of sheaves with small
  complexity.
\end{abstract}

\maketitle

\section{Introduction}

Interesting arithmetic objects often appear in countable sets that can
be naturally partitioned into increasing finite subsets. The
estimation of the cardinality of these subsets is often both
fascinating and important in applications. Well-known examples include
the counting function for primes, the counting function of zeros of
$L$-functions over number fields, or the counting function of
automorphic forms of certain types.
\par
We consider here a similar counting problem where the objects of
interests are certain $\ell$-adic sheaves on a smooth curve over a
finite field, or (more or less) equivalently, certain $\ell$-adic
Galois representations over function fields. In that case, it is not
obvious how to construct finite subsets, even before asking how large
they could be. However, it was shown by
Deligne~\cite{deligne-drinfeld}, as explained by Esnault and
Kerz~\cite[Th. 2.1, Remark 2.2]{esnault-kerz}, that there is, for any
smooth separated scheme $X$ of finite type over a finite field $k$, a
natural notion of ``bounded ramification'' such that the number of
irreducible lisse \'etale $\bar{\Qq}_{\ell}$-sheaves on $X$ is finite,
up to twist by geometrically trivial characters. The problem of saying
more about the order of these finite sets is then the subject of
remarkable conjectures of Deligne in the case of curves predicting,
for suitably restricted ramification, a formula similar to that for
the number of points of an algebraic variety over a finite field in
terms of Weil numbers of suitable weights. This is motivated by the
result of Drinfeld~\cite{drinfeld} computing the number of unramified
$2$-dimensional representations for a projective curve, and showing it
is of this form; see again the survey in~\cite[\S 8]{esnault-kerz} and
the paper of Deligne and Flicker~\cite[\S 6]{deligne-flicker} (and the
lecture~\cite{deligne-newton} of Deligne).
\par
Our goal in this note is relatively modest. We will only consider the
case of curves, and our main result is an explicit upper bound for the
size of certain sets of (isomorphism classes of) \'etale sheaves. We
do not address the crucial issue of lower bounds, but the argument is
quite short and the fact that it uses ideas from spherical codes is
quite appealing. Moreover, the bounds for spherical codes that are
used do not seem to be present in the literature. Furthermore, as we
will explain below, there are natural applications of the estimates we
obtain.
\par
Let $p$ be a prime number and let $k$ be a finite field of
characteristic $p$. Fix an auxiliary prime $\ell\not=p$.  Let $X/k$ be
a smooth geometrically connected algebraic curve over $k$, and let $Y$
be its smooth compactification, with genus $g\geq 0$.
\par
We will consider \emph{middle-extension sheaves} on $X/k$, in the
sense of~\cite{katz-gkm}, i.e., constructible
$\bar{\Qq}_{\ell}$-sheaves $\sheaf{F}$ on $X/k$ such that, for any
open set $U$ on which $\sheaf{F}$ is lisse, with open immersion
$j\,:\, U\injecte X$, we have
$$
\sheaf{F}\simeq j_*j^*\sheaf{F}.
$$
\par
Slightly more concretely, we see that such a sheaf has a largest open
subset $U$ on which it is lisse (defined by the condition that the
stalk be of generic rank), and is determined by its restriction to
this open set. On $U$, $\sheaf{F}$ corresponds uniquely to a
continuous $\ell$-adic representation $\rho$ of the \'etale
fundamental group $\pi_1(U,\bar{\eta})$, defined with respect to some
geometric generic point $\bar{\eta}$ of $U$. As in~\cite[\S
7]{katz-esde}, the middle-extension sheaf $\sheaf{F}$ is called
\emph{pointwise pure of weight $0$} if its restriction to $U$ is
pointwise pure of weight $0$, i.e., the eigenvalues of the local
Frobenius automorphisms at points of $U$ are algebraic numbers, all
conjugates of which have modulus $1$. Furthermore, $\sheaf{F}$ is
called \emph{irreducible} (resp. \emph{geometrically irreducible}) if
$\rho$ is an irreducible representation of the fundamental group
$\pi_1(U,\bar{\eta})$ (resp. of the geometric fundamental group
$\pi_1(U\times\bar{k},\bar{\eta})$).
\par
The collection of middle-extension sheaves on $X/k$ is infinite. We
will measure the complexity of a sheaf over a finite field by its
(analytic) conductor, in order to obtain a well-defined counting
problem. Note that this is a much rougher invariant than that used in
the counting conjectures of Deligne, but it is enough to obtain
finiteness, and the argument below does not seem to allow us to get
any improvement by fixing, for instance, the local monodromy
representations at the missing points for sheaves lisse on a fixed
open set of $X$.
\par
Let $\sheaf{F}$ be a middle-extension sheaf on $X/k$, of rank
$\rank(\sheaf{F})$, with singularities at the finite set
$\sing(\sheaf{F})\subset Y(\bar{k})$. We define the \emph{analytic
  conductor} (often just called \emph{conductor}) of $\sheaf{F}$ to be
\begin{equation}\label{eq-conductor}
  \cond(\sheaf{F})=g(Y)+\rank(\sheaf{F})+
  \sum_{x\in \sing(\sheaf{F})}{\max(1,\swan_x(\sheaf{F}))},
\end{equation}
where $g(Y)$ is the genus of $Y\times\bar{k}$.
\par
Now, for a finite field $k$ and a curve $X/k$ as above, we denote by
$\mext_X(k)$ the category of geometrically irreducible
middle-extension sheaves $\sheaf{F}$ on $X/k$ which are pointwise pure
of weight $0$, and for $c\geq 1$, we denote by $\mext_X(k,c)$ the
subcategory of those that satisfy
$$
\cond(\sheaf{F})\leq c.
$$
\par
We denote also by $\mexti_X(k)$ (resp. $\mexti_X(k,c)$) the set of
\emph{geometric isomorphism classes} of sheaves in $\mext_X(k)$
(resp. in $\mext_X(k,c)$). Our results are bounds for the size of
these sets:

% . Here is a first version, with notation and assumptions as
% above: 

\begin{theorem}\label{th-counting}
  There exist absolute constants $B>0$, $C\geq 1$ such that, with
  notation and assumptions as above, we have
$$
|\mexti_X(k,c)|\leq C|k|^{Bc^6},
$$
for all finite fields $k$ with $|k|\geq 1265c^9$. In particular, for
fixed $c$, we have 
$$
|\mexti_{\Aa^1}(k,c)|=O(|k|^{Bc^6}).
$$
\end{theorem}

\begin{remark} 
(1)  For a fixed $c$, this upper bound is polynomial as a function of
  $q=|k|$.  One can prove lower-bounds which show that this is
  qualitatively correct. For instance, for $X=\Aa^1$, one gets using
  Artin-Schreier sheaves (see Section~\ref{sec-comments}) that
$$
|\mexti_{\Aa^1}(k,c)|\geq |k|^{c/2-1}
$$
for any finite field $k$ and any $c\geq 2$. P. Deligne and F. Jouve
independently explained to us how to improve the exponent $c/2-1$ to
$c-2$ using all rank $1$ sheaves, and it seems an interesting problem
to improve this using other constructions of sheaves of various type
(e.g., those studied by Katz in~\cite{katz-esde}).
\par
(2) We will in fact give fully explicit inequalities, and not just
asymptotic statements, and we can refine the exponent $c^6$ a
little bit (see Proposition~\ref{pr-count-lisse} for these more
precise results).  It is unclear to us what is the best possible
upper-bound achievable by the method of spherical codes that we use:
we do not know what is the right order of magnitude for the quantity
estimated in Theorem~\ref{th-kl} below.
\end{remark}

As far as we know, Theorem~\ref{th-counting} is the first explicit
bound for this type of questions without much stronger restrictions
(e.g., on the rank). One can approach the counting problems by
applying the global Langlands correspondence over function fields (as
proved by Lafforgue~\cite{lafforgue}) to reduce to counting
automorphic forms or representations, and this is indeed how Deligne
and Flicker~\cite{deligne-flicker} proceed to obtain a
``Lefschetz-type'' formula for the counting function for cases where
the local monodromy is unipotent. One might hope to derive
upper-bounds for a fixed rank by means of some version of the Weyl Law
for the distribution of Laplace eigenvalues, but controlling these
estimates when the rank varies seems quite a difficult problem.
\par
The basic idea of the proof is quite simple (and has been known, at
least with respect to showing finiteness, to Deligne\footnote{\ We
  thank H. Esnault for this information.} and to Venkatesh): we first
show that, for $|k|$ large enough, it is enough to count the
\emph{trace functions}
$$
\frfn{\sheaf{F}}{k}\,:\ 
\begin{cases}
  X(k)\lra \bar{\Qq}_{\ell}\\
  x\mapsto \Tr(\frob_{k}\mid \sheaf{F}_{\bar{x}})
\end{cases}
$$
(giving the trace of the geometric Frobenius automorphism of $k$
acting on the stalk of $\sheaf{F}$ at a geometric point $\bar{x}$ over
$x\in X(k)$, seen as a finite-dimensional representation of the Galois
group of $k$) of $\sheaf{F}\in \mext_X(k,c)$. We view these trace
functions (via some isomorphism $\iota\,:\, \bar{\Qq}_{\ell}\lra \Cc$)
as elements of the finite dimensional Hilbert space $C_X(k)$ of
complex-valued functions on $X(k)$, and then see that Deligne's
general form of the Riemann Hypothesis implies that these trace
functions form a ``quasi-orthonormal'' system. In particular, given
that the conductor is $\leq c$, the angle between any two different
trace functions of sheaves in $\mext_X(k,c)$ which are not
geometrically isomorphic is at least $\pi/2-O(1/\sqrt{|k|})$. This
means that the trace functions of sheaves in $\mexti_X(k,c)$ form what
is known as a \emph{spherical code} with this angular separation. This
fact immediately implies that the corresponding set is finite, but
furthermore, we are in a range of spherical codes where one can use
methods of Kabatjanskii and Levenshtein~\cite{kl} (see
also~\cite{levenstein} and~\cite[Ch.  9]{cs}) to derive the
polynomial-type upper bounds of Theorem~\ref{th-counting}. We did not
find the statements for bounds on spherical codes in this range, but
these turn out to be relatively easy to derive from the general
techniques of Kabatjanskii and Levenshtein, as we present in
Section~\ref{sec-prelim} (and they might be of independent interest).
\par
An application of Theorem~\ref{th-counting}, applied to the special
case $X=\Aa^1$, concerns the problem of writing a function defined on
a finite field as a short linear combination of trace functions of
sheaves. Our earlier results in~\cite{fkm, fkm2} show that functions
with such a decomposition can be used in many arguments of analytic
number theory. It is therefore conceptually interesting to show that
such functions are still rather rare: ``most'' functions do not have
such a good decomposition. To make this precise, following~\cite{fkm},
we define \emph{trace norms}:

\begin{definition}[Trace norms]
  Let $s\geq 0$ be a real number. Let $k$ be a finite field of
  characteristic $p$ and let $C(k)$ be the vector space of
  complex-valued functions on $k$. Fix $\ell\not=p$ and an isomorphism
  $\iota\,:\, \bar{\Qq}_{\ell}\lra \Cc$. For $\varphi\in C(k)$, let
$$
\tnorm{\varphi}{s}=\inf\Bigl\{
\sum_{i}{|\lambda_i|\cond(\sheaf{F}_i)^s}+
\sum_{j}|\mu_j|
\Bigr\}
$$
where the infimum runs over all decompositions
$$
\varphi=\sum_i\lambda_i\frfn{\sheaf{F}_i}{k}+\sqrt{|k|}\sum_j \mu_j
\delta_{a_j}
$$
where the sums are finite, $\lambda_i$, $\mu_j$ are complex numbers,
$\sheaf{F}_i$ is an object of $\mext_{\Aa^1}(k)$ and, for any $a\in k$, we
denote by $\delta_{a}$ the delta function at $a$, taking value $1$ at
$0$ and taking value $0$ elsewhere.
\end{definition}

Thus $\tnorm{\cdot}{s}$ is a norm on $C(k)$ (although it seems to
depend on $\ell$ and $\iota$, this will not be of any importance for
us). Using the tautological expansion
$$
\varphi=\frac{1}{\sqrt{|k|}}\sqrt{|k|}\sum_{x\in
  k}{\varphi(x)\delta_x},
$$
we get an immediate upper-bound
\begin{equation}\label{eq-trivial-norm}
  \tnorm{\varphi}{s}\leq |k|^{-1/2}\sum_{x\in k}{|\varphi(x)|}=
  |k|^{1/2}\|\varphi\|_{1}
\end{equation}
where
$$
\|\varphi\|_1=\frac{1}{|k|}\sum_{x\in k}{|\varphi(x)|}
$$
is the $L^1$-norm. This inequality means that $\|i_s\|\leq |k|^{1/2}$,
where $i_s$ is the identity map
$$
i_s\,:\, (C(k),\|\cdot\|_1)\lra (C(k),\tnorm{\cdot}{s}\Bigr).
$$
\par
This is in fact close to the truth, as we show in
Section~\ref{sec-random}: 

\begin{theorem}\label{th-norm}
  Let $k$ be a finite field and let $C(k)$ be the vector space of
  complex-valued functions on $k$. Fix $\ell$ and an isomorphism
  $\iota\,:\, \bar{\Qq}_{\ell}\lra \Cc$ to define the trace norms
  $\tnorm{\cdot}{s}$. Let $i_s$ be the identity map as above. For
  $s\geq 6$ and $|k|$ large enough, we have $\|i_s\|\gg
  \frac{|k|^{1/2}}{\log |k|},$ where the implied constant is absolute.
\end{theorem}
 
Although this is not surprising, we view this as a first basic step in
understanding the properties of functions in $C(k)$ which have good
decompositions in trace functions (an issue that was raised for
instance by P. Sarnak, and which is partly motivated by ``higher-order
Fourier analysis'', in the sense of Gowers and Tao.)

\subsection*{Acknowledgments} 

We wish to thank N. Katz for discussions surrounding these problems,
and H. Esnault for clarifying Deligne's work and its fascinating
general context. We also thank the referee for suggesting that we
treat the case of sheaves on arbitrary curves over finite fields and
for other insightful comments.  Finally, thanks to F. Jouve for
discussions concerning the issue of lower bounds.

\subsection*{Notation}

As usual, $|A|$ denotes the cardinality of a set $A$, and we write
$e(z)=e^{2i\pi z}$ for any $z\in\Cc$.  We write $\Fp=\Zz/p\Zz$.
\par
The notation $f\ll g$ for $x\in A$, or $f=O(g)$ for $x\in A$, where $A$ is an
arbitrary set on which $f$ is defined, are synonymous.
%  mean synonymously that there
% exists a constant $C\geq 0$ such that $|f(x)|\leq Cg(x)$ for all $x\in
% A$. The ``implied constant'' refers to any value of $C$ for which this
% holds. It may depend on the set $A$, which is usually specified
% explicitly, or clearly determined by the context. We write $f(x)\asymp
% g(x)$ to mean $f\ll g$ and $g\ll f$.
\par
% For a constructible sheaf $\sheaf{F}$ on $\Aa^1/k$, and $h\in k$, we
% write $[+h]^*\sheaf{F}$ for the pullback of $\sheaf{F}$ under the map
% $x\mapsto x+h$.
For any algebraic variety $X/k$, any finite extension $k'/k$
and $x\in X(k')$, we denote by $\frtr{\sheaf{F}}{k'}{x}$ the value at
$x$ of the trace function of some $\ell$-adic (constructible) sheaf
$\sheaf{F}$ on $X/k$. We will write $\frfn{\sheaf{F}}{k'}$ for the
function $x\mapsto \frtr{\sheaf{F}}{k'}{x}$ defined on $X(k')$.
\par
We will always assume that some isomorphism $\iota\,:\,
\bar{\Qq}_{\ell}\lra \Cc$ has been chosen and we will allow ourselves
to use it as an identification. Thus, for instance, by
$|\frtr{\sheaf{F}}{k}{x}|^2$, we will mean
$|\iota(\frtr{\sheaf{F}}{k}{x})|^2$.

% If $\sheaf{F}$ is a middle-extension sheaf on
% $\Aa^1/k$, we also write $\dual(\sheaf{F})$ for the middle-extension
% dual of $\sheaf{F}$, i.e., given a dense open set $j\,:\, U\injecte
% \Aa^1$ where $\sheaf{F}$ is lisse, we have
% $$
% \dual(\sheaf{F})=j_*((j^*\sheaf{F})'),
% $$
% where the prime denotes the lisse sheaf of $U$ associated to the
% contragredient of the representation of the fundamental group of $U$
% which corresponds to $j^*\sheaf{F}$ (see~\cite[7.3.1]{katz-esde}). If
% $\sheaf{F}$ is pointwise pure of weight $0$, it is known that
% $$
% \frtr{\dual(\sheaf{F})}{k'}{x}=\overline{\frtr{\sheaf{F}}{k'}{x}}
% $$
% for all finite extensions $k'/k$ and all $x\in k'$ (this property is
% obvious for $x\in U(k')$ and the point is that this extends to the
% singularities when the dual is suitably defined.)  Note for instance
% that $\cond(\dual(\sheaf{F}))=\cond(\sheaf{F})$.

\section{Spherical codes}\label{sec-prelim}

The range of angles defining spherical codes for which we need bounds
is not standard, and we haven't found a direct statement of the exact
form we need in the literature. We therefore first explain how to use
the Kabatjanskii--Levenshtein bounds~\cite{kl} to obtain what we want,
referring to~\cite{levenstein} which is a more accessible reference.
\par
Following the notation in~\cite{levenstein}, we denote by
$M(n,\varphi)$ the largest cardinality of a subset $X\subset
\mathbf{S}^{n-1}$, the $(n-1)$-dimensional unit sphere of the
euclidean space $\Rr^n$ (with inner product $\langle
\cdot,\cdot\rangle_{\Rr}$) which satisfies
$$
\langle x,y\rangle_{\Rr}\leq \cos\varphi
$$
for all $x\not=y$ in $X$.

\begin{theorem}[Polynomial Kabatjanskii--Levenshtein]\label{th-kl}
  Let $\gamma>0$ be a fixed real number. For
\begin{equation}\label{eq-spacing}
  \cos\varphi\leq \frac{\gamma}{\sqrt{n}},
\end{equation}
assuming $n\geq 2\gamma\lceil(\gamma+1)^2\rceil$, we have
$$
M(n,\varphi)\leq
\frac{(n-1)^{\gamma^2+2\gamma+3}}{\Gamma(\gamma^2+2\gamma+2)}.
$$
\end{theorem}

\begin{proof}
  By~\cite[(6.24)]{levenstein}, we have
\begin{equation}\label{eq-kl-bound}
M(n,\varphi)\leq 2\binom{n-1+k}{k}
\end{equation}
for any integer $k\geq 2$ such that
$$
\cos\varphi\leq t_k^{1,1},
$$
where $t_k^{1,1}=p_k^{(n-1)/2,(n-1)/2}$ denotes the largest root of a
certain Gegenbauer polynomial. Furthermore,
by~\cite[(6.25)]{levenstein}, we have
$$
t_k^{1,1}\geq \Bigl(\frac{2(n+k-2)}{(n+2k-2)(n+2k-4)}\Bigr)^{1/2}h_k,
$$
where $h_k$ is the largest root of the $k$-th Hermite polynomial $H_k$
(see also~\cite[Cor. 5.17]{levenstein}; all these polynomials have
only real roots). 
\par
It is known, using elementary arguments (see~\cite[(6.2.14)]{szego}),
that
$$
h_k\geq
\sqrt{\frac{k-1}{2}}.
$$
\par
Under our assumption~(\ref{eq-spacing}), we therefore see
that~(\ref{eq-kl-bound}) holds for $k\geq 2$ such that
$$
\frac{\gamma}{\sqrt{n}}\leq \sqrt{\frac{k-1}{2}}\Bigl(
\frac{2(n+k-2)}{(n+2k-2)(n+2k-4)}
\Bigr)^{1/2}.
$$
\par
Writing $\kappa=k-1$, we see that this certainly holds provided
$$
\gamma^2\leq \frac{\kappa n^2}{(n+2\kappa)^2}
=\frac{\kappa}{(1+2\kappa/n)^2}.
$$
\par
If we assume that $2\kappa /n \leq \gamma^{-1}$, we can take
$\kappa=\lceil (\gamma+1)^2\rceil$, i.e, $k=1+\lceil
(\gamma+1)^2\rceil$. The condition on $\kappa$ translates then to
$$
n\geq 2\gamma\lceil (\gamma+1)^2\rceil,
$$
as stated in the theorem, and we obtain the conclusion
from~(\ref{eq-kl-bound}) using the trivial estimate
$$
2\binom{n-1+k}{k}\leq 2\frac{(n-1)^k}{k!}\leq \frac{(n-1)^k}{(k-1)!}.
$$
\end{proof}

\begin{remark}
  (1) We can improve a bit the result as $k\ra +\infty$ by using the
  asymptotic behavior of the zero $h_k$ of the Hermite polynomial. For
  instance, it is known (see, e.g.,~\cite[(6.32.8)]{szego}, where $k$
  is replaced by $n$ and $h_k$ is denoted $x_1$) that
$$
h_k=\sqrt{2k}-\frac{i_1}{\sqrt[3]{6}}\frac{1}{(2k)^{1/6}}+o(k^{-1/6})
$$
in terms of the first zero $i_1=3.3721\ldots>0$ of the function
$$
\mathrm{A}(x)=\frac{\pi}{3}\sqrt{\frac{x}{3}}\Bigl\{
J_{1/3}\Bigl(2\Bigl(\frac{x}{3}\Bigr)^{3/2}\Bigr)+
J_{-1/3}\Bigl(2\Bigl(\frac{x}{3}\Bigr)^{3/2}\Bigr)
\Bigr\}
$$
(see~\cite[\S 1.81]{szego}; this function is closely related to the
Airy function.)
\par
(2) The point of this result is the polynomial growth of
$M(n,\varphi)$ as $n$ tends to infinity for a fixed $\gamma$, although
it may also be interesting in some ranges when $\gamma$ grows with
$n$. When $\gamma<1$, a bound of this type follows from the early
result of Delsarte, Goethals and Seidel~\cite[Example 4.6]{dgs}. In
contrast, it is known that $M(n,\varphi)$ is bounded independently of
$n$ if $\varphi$ is a fixed angle $>\frac{\pi}{2}$, and grows
exponentially if $\varphi$ is fixed and $<\frac{\pi}{2}$. What is
usually called the Kabatjanskii--Levenshtein bound is an estimate for
the exponential rate of growth in that case
(\cite[Th. 6.7]{levenstein}), which corresponds to $\gamma$ of size
$\alpha n^{1/2}$ for some fixed $\alpha>0$.
\par
(3) In this respect, one can weaken the lower bound $n\geq
2\gamma\lceil(\gamma+1)^2\rceil$ at the cost of a worse exponent of
$n$ in the estimate. This might also be useful, e.g., in a range where
$\gamma\approx n^{\delta}$ for $1/3\leq \delta<1/2$, where the
Kabatjanskii--Levenshtein bound itself does not apply.
\par
(4)   See the paper~\cite{hv} of Helfgott and Venkatesh for other subtle
  applications of the bounds of Kabatjanskii and Levenshtein to
  number-theoretic problems. For an application in analysis that also
  involves quasi-orthogonality, see the paper~\cite{jp} of Jaming and
  Powell.
\par
(5) See T. Tao's blog post
\url{terrytao.wordpress.com/2013/07/18/a-cheap-version-of-the-kabatjanskii-levenstein-bound-for-almost-orthogonal-vectors/}
for a short proof of a slightly weaker, but still polynomial, bound.
\end{remark}

\section{Proof of the main result} 

Throughout this section, we consider a finite field $k$ and a
smooth geometrically connected curve $X/k$ with compactification $Y/k$
of genus $g$.
\par
The proof of Theorem~\ref{th-counting} is based on estimates for
certain subsets of $\mexti_X(k,c)$, which are of independent interest
(in particular, they are more closely related to those considered by
Drinfeld and Deligne, and Esnault--Kerz, Deligne--Flicker.)
\par
Let $U/k$ be a dense open subset of $X/k$. We denote by
$\lisse(U/k,c)$ the category of lisse $\ell$-adic sheaves $\sheaf{F}$
on $U/k$ which are geometrically irreducible on $U$, pointwise pure of
weight $0$, \emph{primitive} in the sense that $U$ is the largest open
set of lissit\'e of the middle-extension sheaf $j_*\sheaf{F}$ on
$X/k$, where $j\,:\, U\injecte X$ is the open embedding of $U$ in $X$,
and with
$$
\cond(j_*\sheaf{F})\leq c.
$$
\par
We denote by $\lissei(U/k,c)$ the set of geometric isomorphism classes of
objects in $\lisse(U/k,c)$, and we further denote by $\lisse_r(U/k,c)$
(resp. $\lissei_r(U/k,c)$) the subcategory where the rank is $\leq r$
(resp. the set of geometric isomorphism classes of this subcategory).
\par
% \par
Our basic estimates are the following:

\begin{proposition}[Counting lisse sheaves]\label{pr-count-lisse}
  Let $k$, $X$ and $Y$ be as above. Let $c\geq 1$.  For any dense open
  set $U/k\injecte X/k$ with $n(U)=|(Y-U)(\bar{k})|\leq c$ and for
  $r\leq c$, we have
\begin{align*}
  |\lissei_r(U/k,c)|&\leq
  \frac{(2|U(k)|)^{90c^2r^4+6\sqrt{10}cr^2+3}}{\Gamma(90c^2r^4)}
\end{align*}
provided $|k|\geq 1500c^3r^6$.
\end{proposition}

This implies Theorem~\ref{th-counting} as follows: a middle-extension
sheaf $\sheaf{F}$ is uniquely determined by its restriction to its
unique largest open dense subset of lissit\'e, and the complement of
such an open set has at most $\cond(\sheaf{F})$ points in
$Y(\bar{k})$, so that
\begin{equation}\label{eq-split}
  |\mexti_X(k,c)|\leq \sum_{n(U)\leq c}{|\lissei_c(U/k,c)|}
\end{equation}
where the sum runs over all open subsets $U/k$ of $X/k$ which are
defined over $k$ and satisfy $n(U)=(|Y-U|)(\bar{k})\leq c$. The number
of terms in this sum is at most $c(q^{c/2}+g)^2$ (indeed, each of the
$\leq c$ missing points is defined over an extension $k_d$ of $k$ of
degree $d\leq c$, and by the Riemann Hypothesis for $Y$, there are at
most $q^d+2gq^{d/2}+1\leq (q^{d/2}+g)^2\leq (q^{c/2}+g)^2$ points on
$Y(k_d)$.) Applying Proposition~\ref{pr-count-lisse} to each $U$, and
noting that the condition on $|k|$ implies $g\leq c\leq |k|^{1/3}$,
and hence also $|U(k)|\leq |k|+2g\sqrt{|k|}+1\leq 4|k|$, the bound of
Theorem~\ref{th-counting} follows.

% for any $\sheaf{F}$
% in some $\lisse(U/k,c)$, the direct image $j_*\sheaf{F}$ is a
% geometrically irreducible middle-extension sheaf on $X/k$, pointwise
% of weight $0$, i.e., an object in $\mext_X(k,c)$. Moreover, since a
% middle-extension sheaf $\sheaf{F}$ is uniquely determined by its
% restriction to its unique largest open dense subset of lissit\'e, and
% the complement of such an open set has at most $\cond(\sheaf{F})$
% points, we can write
% \begin{equation}\label{eq-split}
%   |\mexti_X(k,c)|\leq \sum_{n(U)\leq c}{|\lissei_c(U/k,c)|}
% \end{equation}
% where the sum runs over all open subsets $U/k$ of $X/k$ which are
% defined over $k$ and satisfy $n(U)=(|Y-U|)(\bar{k})\leq c$.
% %% Finish!
% This sum can be parameterized by a subset of the $k$-rational
% effective divisors of degree $\leq c$ on $Y$. Using the Riemann
% hypothesis for the Jacobian of $Y$, it follows that there are at most
% $(q+1)^{c+g}$ terms in the sum.
% \par
% We also introduce the subcategory $\lisse^t(U/k,c)$ of $\lisse(U/k,c)$
% consisting of sheaves which are everywhere tamely ramified, and the
% set $\lissei^t(U/k,c)$ of geometric isomorphism classes, and similarly
% for $\lisse_r^t(U/k,c)$ and $\lissei_r^t(U/k,c)$.

\begin{remark}
  Applying the ``automorphic side to Galois side'' part of the global
  Langlands correspondence on $Y/k$~\cite[Th\'eor\`eme,
  (i)]{lafforgue}, this gives the same upper bound for the number of
  cuspidal automorphic representations of $\GL_r(\Aa_F)$ which are
  unramified on $U$, where $\Aa_F$ is the ring of ad\`eles of the
  function field $F=k(Y)$ of $Y/k$. Even with automorphic techniques,
  it is not clear how to prove such a bound.
\end{remark}

We now start the proof of Proposition~\ref{pr-count-lisse} with a
variant of the well-known upper bounds on the dimension of cohomology
groups of lisse sheaves on algebraic curves.

\begin{lemma}
  Let $k$, $X$ and $Y$ be as above.  Let $j\,:\, U/k\injecte X/k$ be
  the open embedding of a dense open subset with
  $n(U)=|(Y-U)(\bar{k})|$ missing points, and let $\sheaf{F}_1$,
  $\sheaf{F}_2$ be lisse $\ell$-adic sheaves on $U/k$ of rank $r_1$
  and $r_2$, respectively, which are geometrically irreducible. Let
  $c=\max(\cond(j_*\sheaf{F}_1),\cond(j_*\sheaf{F}_2))$. We have
$$
\dim H^1_c(U\times\bar{k},\sheaf{F}_1\otimes\check{\sheaf{F}}_2)\leq
(2c+n(U))r_1r_2.
$$
\end{lemma}

\begin{proof}
  Let $\sheaf{F}=\sheaf{F}_1\otimes\check{\sheaf{F}}_2$, and denote
  $r_i=\rank\sheaf{F}_i$.  Since $H^0_c(U\times\bar{k},\sheaf{F})=0$
  (this is true for all lisse sheaves on $U$), we have
$$
\dim H^1_c(U\times\bar{k},\sheaf{F})
=-\chi_c(U\times\bar{k},\sheaf{F})+ \dim
H^2_c(U\times\bar{k},\sheaf{F}).
$$
\par
The second term is at most $1$ by Schur's Lemma, since
$H^2_c(U\times\bar{k},\sheaf{F})$ is the coinvariant of the generic
geometric fiber under the action of the geometric fundamental group,
and since $\sheaf{F}_1$ and $\sheaf{F}_2$ are geometrically
irreducible. 
\par
Now the Euler-Poincar\'e formula of Grothendieck--Ogg--Shafarevich
(see, e.g.,~\cite[Ch. 2]{katz-gkm}) gives
\begin{align*}
  -\chi_c(U\times\bar{k},\sheaf{F})&=
  -\chi_c(U\times\bar{k})\rank(\sheaf{F})+
  \sum_{x\in\sing(\sheaf{F})}{\swan_x(\sheaf{F})} \\
  &= (n(U)+2g-2)r_1r_2+ \sum_{x\in
    (Y-U)(\bar{k})}{\swan_x(\sheaf{F})}.
\end{align*}
\par
% In the tame case, the second term vanishes, and since the rank is at
% least $1$, we get
% $$
% \dim H^1_c(U\times\bar{k},\sheaf{F})
% \leq 1+(n(U)-2)\rank(\sheaf{F})
% \leq n(U)\rank(\sheaf{F}).
% $$
% \par
% In the general case, 
We have
$$
\swan_x(\sheaf{F})\leq
\rank(\sheaf{F})\lambda_x(\sheaf{F})=r_1r_2\lambda_x(\sheaf{F})
$$
at each $x\in (Y-U)(\bar{k})$, where $\lambda_x(\sheaf{F})$ is the
largest break of $\sheaf{F}$ at $x$. Since
$$
\lambda_x(\sheaf{F})
\leq \max(\lambda_x(\sheaf{F}_1),\lambda_x(\sheaf{F}_2))
\leq \lambda_x(\sheaf{F}_1)+\lambda_x(\sheaf{F}_2),
$$
we get the upper bound
\begin{align*}
  \sum_{x\in (Y-U)(\bar{k})}{\swan_x(\sheaf{F})} &\leq
  \rank(\sheaf{F}) \sum_{x\in (Y-U)(\bar{k})}{
    (\lambda_x(\sheaf{F}_1)+\lambda_x(\sheaf{F}_2))
  }\\
  &\leq r_1r_2\sum_{x\in (Y-U)(\bar{k})}(\swan_x(\sheaf{F}_1)+
\swan_x(\sheaf{F}_2)).
\end{align*}
\par
It follows that
$$
\dim H^1_c(U\times\bar{k},\sheaf{F})\leq 1+r_1r_2(2c+n(U)-2)\leq
(2c+n(U))r_1r_2.
$$
\end{proof}

\begin{remark}
  One might be tempted to estimate $n(U)$ by $c$, but we allow the
  possibility that the sheaves be unramified at some of the points in
  $Y-U$ in this statement (i.e., they are not necessarily primitive),
  in which case an estimate $n(U)\leq c$ is not always valid.
% \par
% (2) The overall order of magnitude of the bound, namely $\approx
% cr_1r_2$ (assuming that $n(U)\leq c$), can not be significantly
% improved, since in the tame case we have the identity
% $$
% \dim H^1_c(U\times\bar{k},\sheaf{F})=1+(n(U)+2g-2)r_1r_2.
% $$
\end{remark}

Now we invoke the Riemann Hypothesis to obtain
``quasi-orthonormality'' relations for trace functions. We only
consider primitive sheaves on a common open set for simplicity.

\begin{lemma}[Quasi-orthogonality relation]\label{lm-orthogonality}
  Let $k$, $X$ and $U\injecte X$ be as above.  Let $c\geq 1$ be given,
  and let $\sheaf{F}_1$, $\sheaf{F}_2$ be sheaves in $\lisse(U/k,c)$
  with ranks $r_i=\rank(\sheaf{F}_i)$.
% \par
% \emph{(1)} If $\sheaf{F}_1$ and $\sheaf{F}_2$ are tame, we have
% $$
% \Bigl|\frac{1}{|k|}\sum_{x\in U(k)}{\frtr{\sheaf{F}_1}{k}{x}
%   \overline{\frtr{\sheaf{F}_2}{k}{x}}}w
% -\delta(\sheaf{F}_1,\sheaf{F}_2) \Bigr|\leq
% \frac{r_1r_2n(U)}{\sqrt{|k|}}
% $$
% where $\delta(\sheaf{F}_1,\sheaf{F}_2)=1$ if $\sheaf{F}_1$ and
% $\sheaf{F}_2$ are geometrically isomorphic on $U$, and is zero
% otherwise.
% \par
% \emph{(2)} In general, 
\par
\emph{(1)} We
have
$$
\Bigl|\frac{1}{|k|}\sum_{x\in
  U(k)}{|\frtr{\sheaf{F}_1}{k}{x}|^2}-1\Bigr| \leq
\frac{3cr^2}{\sqrt{|k|}}.
$$
\par
\emph{(2)} If $\sheaf{F}_1$ and $\sheaf{F}_2$ are not geometrically
isomorphic, then we have
$$
\Bigl|\frac{1}{|k|}\sum_{x\in U(k)}{\frtr{\sheaf{F}_1}{k}{x}
  \overline{\frtr{\sheaf{F}_2}{k}{x}}} \Bigr|\leq
\frac{3cr_1r_2}{\sqrt{|k|}}.
$$
\end{lemma}

\begin{proof}
  We deal with both cases at the same time by redefining
  $\sheaf{F}_2=\sheaf{F}_1$ in (1). By construction, for all $x\in
  U(k)$, we have therefore
$$
\frtr{\sheaf{F}_1}{k}{x} \overline{\frtr{\sheaf{F}_2}{k}{x}}
=\frtr{\sheaf{F}}{k}{x},
$$
where $\sheaf{F}=\sheaf{F}_1\otimes\check{\sheaf{F}}_2$.  The
Grothendieck-Lefschetz trace formula gives
$$
\sum_{x\in U(k)}{\frtr{\sheaf{F}_1}{k}{x}
  \overline{\frtr{\sheaf{F}_2}{k}{x}}}
= \Tr(\frob_k\mid H^2_c(U\times\bar{k},\sheaf{F}))- \Tr(\frob_k\mid
H^1_c(U\times\bar{k},\sheaf{F})).
$$
\par
Because $\sheaf{F}_1$ and $\sheaf{F}_2$ are geometrically irreducible
and pointwise of weight $0$, we have
$$
\Tr(\frob_k\mid
H^2_c(U\times\bar{k},\sheaf{F}))=\delta(\sheaf{F}_1,\sheaf{F}_2)|k|,
$$
by Schur's Lemma and the coinvariant formula for $H^2_c$, where this
delta symbol is $1$ in case (1) and $0$ in case (2). Moreover, since
$\sheaf{F}$ is also pointwise pure of weight $0$, we have
$$
|\Tr(\frob_k\mid H^1_c(U\times\bar{k},\sheaf{F}))| \leq
 \dim H^1_c(U\times\bar{k},\sheaf{F})\sqrt{|k|}
$$
by Deligne's main result on the Riemann Hypothesis over finite
fields~\cite[Th. 1]{weilii}.  Applying the previous lemma, we obtain
the inequalities stated (since here $c\geq \cond(\sheaf{F}_i)\geq
n(U)$ because the sheaves are in $\lisse(U/k,c)$, hence primitive.)
\end{proof}

We can then easily deduce that sheaves are characterized by their
trace functions on $k$ when the ramification is sufficiently small
(this can be compared with the arguments of Deligne presented
in~\cite[\S 5]{esnault-kerz}).

\begin{corollary}\label{cor-injective}
  Let $k$, $X$ and $Y$ be as above, and let $U\injecte X$ be an open
  dense subset of $X/k$. Let $c\geq 1$ be given.
\par
\emph{(1)} If $\sheaf{F}\in \lisse(U/k,c)$ satisfies
$$
3c(\rank(\sheaf{F}))^2<\sqrt{|k|},
$$
then $\frfn{\sheaf{F}}{k}$ is non-zero on $U(k)$. 
\par
\emph{(2)} If $\sheaf{F}_1$ and $\sheaf{F}_2$ are sheaves in
$\lisse(U/k,c)$ with
$$
3c\rank(\sheaf{F}_1)(\rank(\sheaf{F}_1)+\rank(\sheaf{F}_2))<\sqrt{|k|},
$$
then $\sheaf{F}_1$ and $\sheaf{F}_2$ are geometrically isomorphic if
and only if their trace functions coincide on $U(k)$, up to a fixed
multiplicative constant of modulus $1$.
\par
In particular, if $c\geq 1$ and $r\geq 1$ satisfy $6cr^2<\sqrt{|k|}$,
the map $\sheaf{F}\mapsto \frfn{\sheaf{F}}{k}$ is injective on any set
of representatives of geometric isomorphism classes of objects in
$\lisse_r(U/k,c)$.
% \par
% If the sheaves are tame, the same holds with the conditions above
% replaced by
% $$
% n(U)(\rank(\sheaf{F}))^2<\sqrt{|k|},\quad\quad
% n(U)\rank(\sheaf{F}_1)(\rank(\sheaf{F}_1)+\rank(\sheaf{F}_2))<\sqrt{|k|}.
% $$
% respectively.
\end{corollary}

\begin{proof}
  For (1), it is enough to note that the assumption implies by
  Lemma~\ref{lm-orthogonality} that
$$
\sum_{x\in U(k)}{|\frtr{\sheaf{F}}{k}{x}|^2}>0.
$$
\par
For (2), only the ``only if'' part needs proof (by a well-known property
of geometric isomorphism: the trace functions coincide on $k$ up to a
fixed non-zero scalar). So assume that there exists $\theta\in\Rr$
such that
$$
\frtr{\sheaf{F}_1}{k}{x}=e^{i\theta}\frtr{\sheaf{F}_2}{k}{x}
$$
for all $x\in U(k)$. We then obtain
$$
\Bigl|\frac{1}{|k|}\sum_{x\in U(k)}{
  \frtr{\sheaf{F}_1}{k}{x}\overline{ \frtr{\sheaf{F}_2}{k}{x} } }
\Bigr|= \frac{1}{|k|}\sum_{x\in U(k)}{ |\frtr{\sheaf{F}_1}{k}{x}|^2}
\geq 1-\frac{3c\rank(\sheaf{F}_1)^2}{\sqrt{|k|}}
$$
by Lemma~\ref{lm-orthogonality}. If, by contraposition, the sheaves
were \emph{not} geometrically irreducible, we would get
$$
\Bigl|\frac{1}{|k|}\sum_{x\in U(k)}{
  \frtr{\sheaf{F}_1}{k}{x}\overline{ \frtr{\sheaf{F}_2}{k}{x} } }
\Bigr| \leq
\frac{3c\rank(\sheaf{F}_1)\rank(\sheaf{F}_2)}{\sqrt{|k|}}
$$
by the same lemma, and by comparing we deduce that
$$
\sqrt{|k|}\leq
3c\rank(\sheaf{F}_1)(\rank(\sheaf{F}_1)+\rank(\sheaf{F}_2))
$$
in that case.
\end{proof}

We continue with the data $k$, $X/k$ and $Y/k$ as above. We let $V$
denote the vector space of complex-valued functions $U(k)\lra \Cc$. We
view $V$ as a complex Hilbert space with the inner product
$$
\langle \varphi_1,\varphi_2\rangle=\frac{1}{|k|}\sum_{x\in
  U(k)}{\varphi_1(x)\overline{\varphi_2(x)}},
$$
or as a \emph{real} Hilbert space isomorphic to $\Rr^{2|U(k)|}$ with
coordinates given by 
$$
(\Reel\varphi(x),\Imag\varphi(x))_{x\in
  U(k)},
$$
and with inner product
$$
\langle v,w\rangle_{\Rr}=\frac{1}{|k|}\sum_{i=1}^{2|U(k)|}{ v_iw_i}
$$
for $v$, $w\in \Rr^{2|U(k)|}$.
\par
We have the compatibility
$$
\|\varphi\|=\|\varphi\|_{\Rr}
$$
for $\varphi\in V$, with obvious notation. Similarly, the angle
$\theta_{\Rr}(\varphi_1,\varphi_2)\in [0,\pi[$ between $\varphi_1$,
$\varphi_2\in V$ (viewed as a real Hilbert space) is defined by
$$
\langle\varphi_1,\varphi_2\rangle_{\Rr}= \|\varphi_1\| \|\varphi_2\|
\cos\theta_{\Rr}(\varphi_1,\varphi_2),
$$
and also satisfies
$$
\cos\theta_{\Rr}(\varphi_1,\varphi_2)= \frac{\Reel(\langle
  \varphi_1,\varphi_2\rangle)}{\|\varphi_1\|\|\varphi_2\|}.
$$
\par
Fix now $c\geq 1$ and $r\leq c$. If $|k|>3cr^2$ and $\sheaf{F}\in
\lisse(U/k,c)$ has rank $\leq r$, we can define
$$
v_{\sheaf{F}}=\frac{\varphi}{ \| \varphi\| }
$$
where $\varphi$ is the restriction to $U(k)$ of $\frfn{\sheaf{F}}{k}$,
since the trace function is not identically zero by the previous
corollary. This is a vector on the unit sphere of $V$.

\begin{lemma}[Spherical codes from sheaves]\label{lm-cos}
  With notation as above, for fixed $c\geq 1$ and $r\leq c$ with
  $12cr^2<\sqrt{|k|}$, we have
$$
\cos\theta_{\Rr}(v_{\sheaf{F}_1},v_{\sheaf{F}_2})\leq
\frac{6cr^2}{\sqrt{|k|}} \leq \frac{3\sqrt{10}cr^2}{\sqrt{2|U(k)|}}
$$
for any sheaves $\sheaf{F}_1$ and $\sheaf{F}_2$ in $\lisse(U/k,c)$
which are not geometrically isomorphic and have rank $\leq r$.
\end{lemma}

\begin{proof}
We have
$$
\cos\theta_{\Rr}(v_{\sheaf{F}_1},v_{\sheaf{F}_2})
=\frac{\Reel(\langle \varphi_1,\varphi_2\rangle)}{\|\varphi_1\|
\|\varphi_2\|}
$$
where $\varphi_i$ is the restriction of $\frfn{\sheaf{F}_i}{k}$ to
$U(k)$. By Lemma~\ref{lm-orthogonality}, we have
$$
|\langle \varphi_1,\varphi_2\rangle|
\leq \frac{3cr^2}{\sqrt{|k|}},\quad
\quad
\|\varphi_1\|\|\varphi_2\|\geq 1-\frac{3cr^2}{\sqrt{|k|}}.
$$
\par
Since
$$
\frac{x}{(1-x)}\leq 2x
$$
for $0\leq x\leq 1/4$, and
$$
|U(k)|\leq |k|+2g\sqrt{|k|}+1\leq |k|+3g\sqrt{|k|}\leq
\frac{5}{4}\sqrt{|k|} 
$$
under our assumption $12cr^2\leq |k|^{1/2}$, we get the result.
\end{proof}

It follows directly from this lemma, Corollary~\ref{cor-injective} and
from the definition in Section~\ref{sec-prelim} that for $r\leq c$ and
$12cr^2<\sqrt{|k|}$, we have
$$
|\lissei_r(U/k,c)|\leq
M\Bigl(2|U(k)|,\arccos\Bigl(\frac{3\sqrt{10}cr^2}{\sqrt{2|U(k)|}}
\Big)\Bigr).
$$
\par
We can then apply Theorem~\ref{th-kl} with parameters
$$
(n,\gamma)=(2|U(k)|,3\sqrt{10}cr^2),
$$
and the upper-bound in Proposition~\ref{pr-count-lisse} follows as
soon as the condition $n\geq 2\gamma\lceil (\gamma+1)^2\rceil$ in
Theorem~\ref{th-kl} is satisfied. Since $|U(k)|\geq |Y(k)|-c\geq
|k|-2g\sqrt{|k|}+1-c\geq \frac{5}{6}|k|-c$, this condition is
satisfied provided
$$
\frac{5|k|}{6} \geq
3\sqrt{10}cr^2\left\{(3\sqrt{10}cr^2+1)^2+1\right\}+c,
$$
which holds for $|k|\geq 1265c^3r^6$, a condition which also implies
the previous conditions on $|k|$ from Corollary~\ref{cor-injective}
and Lemma~\ref{lm-cos}.

\section{Comments}\label{sec-comments}

The bounds we have obtained are certainly far from the truth. In fact,
it would be even more interesting to have good lower bounds, but this
question is not currently very well understood. This can be
illustrated with the following two remarks:
\par
(1) (Pointed out by Venkatesh): We do not know if, given a large
enough rank $r\geq 1$, there exists a single unramified cusp form on
$\GL_r(K)$, where $K$ is the function field of a fixed curve over a
finite field of genus $>1$; in our notation, the question is, given an
open dense set $U\subset \Aa^1$ defined over $k$, whether there exists
\emph{some} geometrically irreducible lisse sheaf $\sheaf{F}$ on $U$
for \emph{every} large enough rank $r\geq 1$.
% Similarly, the referee pointed out that the same question is open
% (in the ramified case) for division algebras over $F$.
\par
(2) (Pointed out by Katz): Deligne and
Flicker~\cite[Prop. 7.1]{deligne-flicker} prove, using automorphic
methods, that there exist $q=|k|$ lisse sheaves on $(\Pp^1-S)/k$,
where $S$ is an \'etale divisor of degree four (e.g., on
$\Pp^1-\{\text{four points in } k\}$) of rank $2$, with ``principal
unipotent local monodromy'' at the singularities (see~\cite[\S
1]{deligne-flicker} for precise definitions.) However, only a bounded
number of such sheaves are explicitly known (bounded as $q$ varies)!
Examples include semistable families of elliptic curves with four
singular fibers, from Beauville's classification~\cite{beauville}.
  % (one of the corresponding sheaves occurs naturally
  % in~\cite[Th. 11.3, App. A]{fkm} as the Fourier transform of the
  % sheaf
% $$
% \mathrm{Sym}^2([x\mapsto x^2]^*\sheaf{K})
% $$
% where $\sheaf{K}$ is the Kloosterman sheaf of rank
% $2$~\cite{katz-gkm}).
\par
We now indicate some examples of families of sheaves which give easy
lower bounds. We denote by $p$ the characteristic of $k$, and we
consider $X=\Aa^1$ for simplicity.
\par
(1) If $U\injecte \Aa^1$ is a dense open subset (defined over $k$),
and $f_1$ (resp. $f_2$) is a regular function $f_1\,:\, U\lra \Aa^1$
(resp. a non-zero regular function $f_2\,:\, U\lra \Gg_m$) both
defined over $k$, one has the Artin-Schreier-Kummer lisse sheaf
$$
\sheaf{F}=\sheaf{L}_{\psi(f_1)}\otimes\sheaf{L}_{\chi(f_2)}
$$
defined for any non-trivial additive character $\psi\,:\, k\lra
\bar{\Qq}_{\ell}^{\times}$ and multiplicative character $\chi\,:\,
k^{\times}\lra \bar{\Qq}_{\ell}^{\times}$, which satisfy
$$
\frtr{\sheaf{F}}{k}{x}=\psi(f_1(x))\chi(f_2(x))
$$
for $x\in U(k)$. These sheaves are all of rank $1$ (in particular,
they are geometrically irreducible) and pointwise pure of weight
$0$. Moreover, possible geometric isomorphisms among them are
well-understood (see, e.g.,~\cite[Sommes Trig. (3.5.4)]{deligne}): if
$(g_1, g_2)$ are another pair of functions we have a geometric
isomorphism
$$
\sheaf{L}_{\psi(f_1)}\otimes\sheaf{L}_{\chi(f_2)}\simeq
\sheaf{L}_{\psi(g_1)}\otimes\sheaf{L}_{\chi(g_2)}
$$
if and only if: (1) $f_1-g_1$ is of the form 
$$
f_1-g_1=h^{p}-h+C
$$
for some regular function $h$ on $U$ and some constant $C\in\bar{k}$;
(2) $f_2/g_2$ is of the form
$$
\frac{f_2}{g_2}=Dh^{d}
$$
where $d\geq 2$ is the order of the multiplicative character $\chi$,
$h$ is a non-zero regular function on $U$ and $D\in\bar{k}^{\times}$. 
\par
Furthermore, the conductor of these sheaves is fairly easy to
compute. The singularities are located (at most) at $x\in
\Pp^1-U$. For each such $x$, the Swan conductor at $x$ is determined
only by $f_1$, and is bounded by the order of the pole of $f_1$ (seen
as a function $\Pp^1\lra \Pp^1$) at $x$, and there is equality if this
order is $<p$.
\par
In particular, if $\chi$ is trivial, the conductor of
$\sheaf{L}_{\psi(f_1)}$ is $\leq 1+\deg(f_1)$. Taking polynomials of
degree $\leq c-1$, modulo constants and modulo polynomials of the form
$h^p-h$ where $\deg(h)\leq \lfloor \frac{c-1}{p}\rfloor\leq c/2$, we
see that we have
$$
|\lissei_1(\Aa^1/k,c)|\geq |k|^{c-1-c/2},
$$
as indicated in the remark after Theorem~\ref{th-counting}. 
\par
(2) The following examples are studied by Katz~\cite[Ex.
7.10.2]{katz-esde}. Let $C/k$ be a smooth projective geometrically
connected algebraic curve, and
$$
f\,:\, C\lra \Pp^1
$$
a non-constant map defined over $k$ which is not a $p$-th power. Let
$D\subset C$ be the divisor of poles of $f$. Let $Z\subset C-D$ be the
set of zeros of the differential $df$, and let $S=f(Z)$ be the set of
singular values of $f$. One says that $f$ is \emph{supermorse} if
$\deg(f)<p$, all zeros of $df$ are simple, and $f$ separates these
zeros (i.e., $|S|=|Z|$). Then, denoting by 
$$
f_0\,:\, C-D\lra \Aa^1
$$
the restriction of $f$ to $C-D$, the sheaf
$$
\sheaf{F}_f=\ker(\Tr\,:\, f_{0,*}\bar{\Qq}_{\ell}\lra
\bar{\Qq}_{\ell})
$$
is an irreducible middle-extension sheaf on $\Aa^1/k$, of rank
$\deg(f)-1$, pointwise pure of weight $0$ and lisse on $\Aa^1-S$ with
$$
\frtr{\sheaf{F}_f}{k}{x}=|\{y\in C(k)\,\mid\, f(y)=x\}|-1
$$
for $x\in k-S$. This sheaf is also everywhere tamely ramified, so its
conductor is $|Z|+\deg(f)-1$. However, it is not obvious how to count
how many sheaves with conductor $\leq c$ one may obtain in this
manner. 
\par
(3) There exists a Fourier transform on middle-extension sheaves on
$\Aa^1/k$, corresponding to the Fourier transform of trace functions,
which was defined by Deligne and developed especially by Laumon;
precisely, consider a middle-extension sheaf $\sheaf{F}$ which is
geometrically irreducible, of weight $0$, and not geometrically
isomorphic to $\sheaf{L}_{\psi}$ for some additive character
$\psi$. Fix a non-trivial additive character $\psi$. Then the Fourier
transform $\sheaf{G}=\ft_{\psi}(\sheaf{F})(1/2)$ (where the Tate twist
is defined after picking the square root of $|k|$ in
$\bar{\Qq}_{\ell}$ mapping to $\sqrt{|k|}>0$ via our chosen $\iota$)
satisfies
$$
\frtr{\sheaf{G}}{k}{t}=-\frac{1}{\sqrt{|k|}}
\sum_{x\in k}{\frtr{\sheaf{F}}{k}{x}\psi(tx)}
$$
for $t\in k$, and it is a middle-extension sheaf, geometrically
irreducible and pointwise pure of weight $0$ (see~\cite[\S
7]{katz-esde} for a survey and details). Moreover, one can show that
the conductor of $\sheaf{G}$ is bounded polynomially in terms of the
conductor of $\sheaf{F}$ (see, e.g.,~\cite[Prop. 7.2]{fkm}).
% In particular, applying the Fourier transform to the previous
% examples, we find many examples of one-parameter families of
% exponential sums arising as trace functions with bounded conductor,
% namely
% $$
% x\mapsto -\frac{1}{\sqrt{|k|}}
% \sum_{y\in (C-D)(k)}{e\Bigl(\frac{x f(y)}{p}\Bigr)}
% $$
% for the sheaves $\sheaf{F}_f$, and
% $$
% x\mapsto -\frac{1}{\sqrt{|k|}}
% \sum_{y\in k}{\chi(f_2(y))\psi(f_1(y)+xy)}
% $$
% for Artin-Schreier-Kummer sheaves.
However, even without inquiring about possible fixed points of the
Fourier transform, its use would at best double any given lower bound
for the number of sheaves with bounded ramification.

\section{Trace norms and random functions}\label{sec-random}

We describe in this section the proof of Theorem~\ref{th-norm}. The
idea is to show that ``random'' functions defined on $k$ have large
trace norms:
% Precisely, we will consider the following model of random functions on
% a finite field $k$: the values $\varphi(x)$, for $x\in k$, are
% independent complex-valued random variables (on some fixed probability
% space $(\Omega,\Sigma,\proba)$), and they are identically distributed
% and satisfy $|\varphi(x)|\leq 1$ for all $x$ and
% $$
% \expect(\varphi(x))=0,\quad\quad \expect(|\varphi(x)|^2)>0.
% $$
% \par
% We will show that $\tnorm{\varphi}{s}$ is close to $\sqrt{|k|}$ with
% high probability if $s\geq 6$. Precisely:

\begin{theorem}\label{th-random-norm}
  Let $X$ be a complex-valued random variable with $\expect(X)=0$,
  $\expect(|X|^2)>0$, $|X|\leq 1$. For $p$ prime, let $\varphi$ be
  random functions in $C(k)$ such that the values $\varphi(x)$ are
  independent and identically distributed like $X$. For any $N\geq 1$,
  there exists $\alpha\geq 1$ depending only on $N$ and on the law of
  $X$, such that we have
$$
\proba\Bigl(\frac{\sqrt{|k|}}{\alpha \log |k|} \leq
\tnorm{\varphi}{s}\leq \sqrt{|k|}\Bigr) =1+O(|k|^{-N}),
$$
for all $s\geq 6$, where the implied constant depends only on $N$ and
on the law of $X$.
\end{theorem}

This result easily implies Theorem~\ref{th-norm}.

\begin{proof}[Proof of Theorem~\ref{th-norm}]
  We must show the existence of a non-zero function $\varphi\in
  C(k)$ such that
$$
\tnorm{\varphi}{s}\gg \frac{\sqrt{|k|}}{\log |k|} \|\varphi\|_1.
$$
\par 
Since $\tnorm{\varphi}{s}\geq \tnorm{\varphi}{6}$, it is enough to do
this for $s=6$, and this follows from Theorem~\ref{th-random-norm}
(for $N=1$, for instance) and the property~(\ref{eq-concentr-1})
proved below.
% that the corresponding random function has this property with very
% high probability as $|k|\ra +\infty$. In particular, as soon as the
% probability is $>0$, it follows that there exists at least one such
% function $\varphi$ (the implied constant is absolute since we fix
% the distribution $X$ for all $|k|$).
\end{proof}

Theorem~\ref{th-random-norm} is a simple probabilistic argument, which
uses little knowledge of trace functions in addition to the counting
result Theorem~\ref{th-counting}. However, it requires some
quantitative upper bound for $|\mext_{\Aa^1}(k,c)|$, and in fact it
requires some control even for $c$ varying with $k$. 
\par
First, we note the following criterion for lower bounds of
$\tnorm{\varphi}{s}$.

\begin{proposition}[Lower bounds for trace norms]\label{pr-lower}
  Let $k$ be a finite field and let $\varphi\in C(k)$ be any
  function. Let $s\geq 1$, $\gamma>0$ and $A\geq 0$ be numbers such
  that
$$
|\varphi(y)|\leq A|k|^{1/2-\gamma}
$$
for all $y\in k$ and
$$
  \Bigl|\sum_{x\in k}{K(x)\varphi(x)}\Bigr|\leq
  A|k|^{1-\gamma}\cond(K)^s,\quad
  \sum_{x\in k}|\varphi(x)|^2\geq A^{-1}|k|
$$
for all trace functions $K=\frfn{\sheaf{F}}{k}$ of sheaves
$\sheaf{F}\in \mext(k)$. Then we have
$$
\tnorm{\varphi}{s}\geq A^{-2}|k|^{\gamma}.
$$
\end{proposition}

% \begin{proof}
% The first two assumptions imply by linearity that
% $$
% \Bigl|\sum_{x\in k}{\varphi(x)K(x)}\Bigr|\leq
% A|k|^{1-\gamma}\tnorm{K}{s}
% $$
% for all $K\in C(k)$. Taking $K=\bar{\varphi}$ and using the last
% assumption, we get
% $$
% A^{-1}|k|\leq A|k|^{1-\gamma}\tnorm{K}{s},
% $$
% hence the result.
% \end{proof}

\begin{remark}
  This result combined with~\cite[Cor. 1.6]{fkm} also allows us to
  obtain concrete examples of functions with large trace
  norms. Precisely, if $k=\Fp$ identified with $\{1,\ldots, p\}$ and
  $\varphi(n)=\rho_f(n)$ for $1\leq  n\leq p$, where
$$
f(z)=\sum_{n\geq 1}{\rho_f(n)n^{(\kappa-1)/2}e(nz)}
$$
is the Fourier expansion of a classical holomorphic cusp form with
weight $\kappa\geq 2$ and level $N\geq 1$ (and trivial nebentypus),
then for any $\eps>0$, we can derive
$$
\tnorm{\varphi}{s}\gg p^{1/8-\eps}
$$
for all $s\geq s_0$ and $p$ large enough, where $s_0$ is some absolute
constant, and where the implied constant depends on $f$ and
$\eps$. % One can, for instance, take $\varphi(n)$ to be
% $n^{-11/2}\tau(n)$ for $1\leq n\leq p$, where $\tau$ is the
% Ramanujan function appearing in the expansion of the discriminant
% cusp form $\Delta(z)$.
The same result holds for the Fourier coefficients of a Maass cusp
form.  It seems quite conceivable that this estimate should in fact be
true with $1/8$ replaced with $1/2$. More generally, it seems to be an
interesting de-randomization problem to construct explicit functions
$\varphi\in C(k)$ (say bounded by $1$) with $\tnorm{\varphi}{s}$ as
large as the value $\approx |k|^{1/2}$ given by
Theorem~\ref{th-random-norm} for random functions.
\end{remark}

We now begin the proof of Theorem~\ref{th-random-norm} with some
probabilistic preliminaries. We recall that a real-valued random
variable $X$ is called \emph{$\sigma$-subgaussian}, for some
$\sigma>0$, if
$$
\expect(e^{tX})\leq \exp\Bigl(\frac{\sigma^2t^2}{2}\Bigr)
$$
for all $t\in\Rr$.  The following properties are easy: (1) if $X$ is
$\sigma$-subgaussian, then 
$$
\proba(|X|\geq \alpha)\leq
2\exp\Bigl(-\frac{\alpha^2}{2\sigma^2}\Bigr)
$$
for all $\alpha\geq 0$, and (2) if $X_1$, \ldots, $X_k$ are
$\sigma_i$-subgaussian and independent and $a_i\in\Rr$, then
$a_1X_1+\cdots+a_kX_k$ is $\sigma$-subgaussian where
$$
\sigma^2=\sum_{i=1}^k{a_i^2\sigma_i^2}.
$$

We will use the following lemma:

\begin{lemma}\label{lm-probas}
  Let $k$ be a finite field. Let $\sigma>0$, and for $x\in k$, let
  $\varphi(x)$ be independent complex-valued random variables with the
  same distribution, such that $|\varphi(x)|\leq 1$,
  $\expect(\varphi(x))=0$, $\expect(|\varphi(x)|^2)=\sigma^2$.
\par
\emph{(1)} The random variables $\Reel(\varphi(x))$ and
$\Imag(\varphi(x))$ are $1$-subgaussian.
% , i.e., for any $t\in\Rr$, we have
% $$
% \expect(e^{t\Reel(\varphi(x))})\leq e^{t^2/2},\quad\quad
% \expect(e^{t\Imag(\varphi(x))})\leq e^{t^2/2}.
% $$
\par
\emph{(2)} There exists $\nu_1$, $\nu_2>0$ and $c_1$ $c_2>0$,
depending only on the common distribution of $\varphi(x)$, such that
\begin{align}
  \proba\Bigl(\sum_{x\in k}|\varphi(x)|^2\geq \nu_1 |k|\Bigr)&\geq
  1-e^{-c_1|k|^2},\label{eq-concentr-2}\\
  \proba\Bigl(\sum_{x\in k}|\varphi(x)|\geq \nu_2 |k|\Bigr)&\geq
  1-e^{-c_2|k|^2}\label{eq-concentr-1}.
\end{align}
\end{lemma}

\begin{proof}
  (1) Since $|\Reel(\varphi(x))|\leq |\varphi(x)|\leq 1$ and
  $\expect(\Reel\varphi(x))=0$ (and similarly for the imaginary
  part), this follows from the fact that if $X$ is a real-valued
  random variable with $\expect(X)=0$ and which satisfies $|X|\leq
  \sigma$, then $X$ is $\sigma$-subgaussian (see, e.g.,~\cite[Example
  1.2]{bk}).
\par
(2) These are elementary instances of concentration of measure (see,
e.g.,~\cite[\S 2.1]{tao}).
\end{proof}

% , but we give the details for
% completeness. Let $t>0$ and $\delta>0$ be such that
% $$
% t=\proba(|\varphi(x)|\geq \sqrt{\delta})>0.
% $$
% \par
% Let then $B_x$ be the Bernoulli random variable equal to $1$ if
% $|\varphi(x)|\geq \sqrt{\delta}$ and $0$ otherwise. The $B_x$ are
% independent and $\proba(B_x=1)=t>0$, and moreover we have
% \begin{gather*}
%   \proba\Bigl(\sum_{x\in k}|\varphi(x)|^2< \demi t\delta |k|\Bigr)\leq
%   \proba\Bigl(\sum_{x\in k}B_x<\demi t |k|\Bigr),\\
%   \proba\Bigl(\sum_{x\in k}|\varphi(x)|< \demi t\sqrt{\delta}
%   |k|\Bigr)\leq \proba\Bigl(\sum_{x\in k}B_x<\demi t |k|\Bigr).
% \end{gather*}
% \par
% The random variables $C_x=B_x-t$ are centered, independent, and
% $|C_x|\leq 1$. Thus the $C_x$ are also $1$-subgaussian and
% independent, and hence by Lemma~\ref{lm-sub-gaussian}, we have
% $$
% \proba\Bigl(\sum_{x\in k}{C_x}\leq -\alpha\Bigr) \leq
% \exp\Bigl(-\frac{\alpha^2}{2|k|}\Bigr).
% $$
% for any $\alpha\geq 0$. Since
% $$
% \proba\Bigl(\sum_{x\in k}B_x<\demi t |k|\Bigr) \leq
% \proba\Bigl(\sum_{x\in k}C_x\leq -\demi t |k|\Bigr)
% $$
% we get
% $$
% \proba\Bigl(\sum_{x\in k}|\varphi(x)|^2< \demi t\delta |k|\Bigr) \leq
% \exp(-c|k|^2)
% $$
% with $c=t/8$, and similarly for the other probability.
% \end{proof}

The next step shows that a random function is, with very high
probability, strongly orthogonal to the trace function of any sheaf
with small conductor:

\begin{lemma}\label{lm-one-ortho}
  Let $k$ be a finite field, $\varphi$ a random function on $k$ as
  above. Let $K=\frfn{\sheaf{F}}{k}$ for some $\sheaf{F}\in
  \mext(k)$. We have
$$
\proba\Bigl(\Bigl| \sum_{x\in k}K(x)\varphi(x) \Bigr|\geq \alpha
\cond(\sheaf{F})^s\sqrt{|k|\log |k|}\Bigr) \leq 8|k|^{-\demi \alpha^2
  \cond(\sheaf{F})^{2s-2}}
$$
for $s\geq 2$ and $\alpha>0$. 
\end{lemma}

\begin{proof}
We write
$$
\varphi(x)=\varphi_1(x)+i\varphi_2(x),\quad\quad
K(x)=K_1(x)+iK_2(x)
$$
the real and imaginary parts of $\varphi(x)$ and $K(x)$. Expanding the
product, we have
$$
\proba\Bigl(\Bigl| \sum_{x\in k}K(x)\varphi(x) \Bigr|\geq \beta\Bigr)
\leq \sum_{1\leq i,j\leq 2}\proba\Bigl(\Bigl|
\sum_{x\in k}K_i(x)\varphi_j(x) \Bigr|\geq \beta/4\Bigr)
$$
for any $\beta\geq 0$. For $i=1$ or $2$, since the real and imaginary
parts $\varphi_j$ of $\varphi(x)$ are $1$-subgaussian and independent,
we get
$$
\proba\Bigl(\Bigl|\sum_{x\in k}{K_i(x)\varphi_j(x)}
\Bigr|\geq \beta)\leq
2\exp\Bigl(-\frac{\beta^2}{2\sigma_i^2}\Bigr)
$$
for $\beta\geq 0$, where
$$
\sigma_i^2=\sum_{x\in k}K_i(x)^2\leq \sigma_K^2=\sum_{x\in
  k}{|K(x)|^2}\leq |k|\cond(\sheaf{F})^2
$$
and we get the result by taking $\beta=\alpha\cond(K)^s\sqrt{|k|\log
  |k|}$.
\end{proof}

We now extend this to all sheaves with small enough conductor:

\begin{lemma}
  Let $k$ be a finite field, $\varphi$ a random function on $k$ as
  above.  For any $\gamma<1/9$, and any $N\geq 1$, there exists
  $\alpha\geq 1$, depending only on $\gamma$ and $N$, such that for
  any finite field $k$, we have
$$
\proba\Bigl( \Bigl| \sum_{x\in k}\frtr{\sheaf{F}}{k}{x}\varphi(x)
\Bigr|\geq \alpha \cond(\sheaf{F})^4\sqrt{|k|\log |k|}, \text{ for
  \emph{some} $\sheaf{F}$ with $\cond(\sheaf{F})\leq
  \frac{4}{10}|k|^{\gamma}$} \Bigr)\ll |k|^{-N}.
$$
\end{lemma}

\begin{proof}
  For any $s\geq 2$, let
\begin{multline*}
  \varpi(c_1,c_2)=\proba\Bigl( \Bigl| \sum_{x\in
    k}\frtr{\sheaf{F}}{k}{x}\varphi(x) \Bigr|\geq \alpha
  \cond(\sheaf{F})^s\sqrt{|k|\log |k|},\\ \text{ for \emph{some} sheaf
    $\sheaf{F}$ in $\mext(k)$ with $c_1\leq\cond(\sheaf{F})\leq c_2$}
  \Bigr).
\end{multline*}
\par
We have
$$
\varpi\leq \sum_{1\leq j\leq \bigl\lceil\gamma\frac{\log |k|}{\log
    2}\bigr\rceil}\varpi(2^{j-1},2^j)\leq \sum_{j\ll \gamma\log
  |k|}|\mexti(k,2^j)|\times 8|k|^{- \alpha^2 2^{(j-1)(2s-2)-1}}
$$
by Lemma~\ref{lm-one-ortho}. All conductors involved are $\leq
\frac{4}{10}|k|^{\gamma}<\left(\frac{1}{1265}\right)^{1/9}|k|^{1/9}$
by assumption, so we deduce
$$
\varpi\ll \sum_{j\ll \gamma\log|k|}|k|^{B2^{6j}- \alpha^2
  2^{(j-1)(2s-2)-1}},
$$
for some absolute constant $B\geq 1$ by
Theorem~\ref{th-counting}. Taking $s=4$, the exponent of $|k|$ is
$$
B2^{6j}- \alpha^2 2^{(j-1)(2s-2)-1} =B2^{6j}- \alpha^2
2^{6j-7}=(B-2^{-7}\alpha^2)2^{6j} \leq -\frac{\alpha^2}{2^{8}}2^{6j}
$$
under the assumption that $\alpha^2\geq 2^{8}B$, and we get $\varpi\ll
(\log |k|)|k|^{-\alpha^2/4}$, which gives the result by taking
$\alpha>0$ large enough.
\end{proof}

\begin{proof}[Proof of Theorem~\ref{th-random-norm}]
  For any function $\varphi\in C(k)$ with $|\varphi|\leq 1$, and any
  sheaf $\sheaf{F}$ in $\mext(k)$, we have trivially
% $$
% \Bigl| \sum_{x\in k}\frtr{\sheaf{F}}{k}{x}\varphi(x) \Bigr|\leq
% \|\frfn{\sheaf{F}}{k}\|_{\infty}|k|\leq \cond(\sheaf{F})|k|
% $$
% and hence if the conductor satisfies
% $$
% \cond(\sheaf{F})>\frac{4}{10}|k|^{1/10},
% $$
$$
\Bigl| \sum_{x\in k}\frtr{\sheaf{F}}{k}{x}\varphi(x) \Bigr|\leq
100\cond(\sheaf{F})^6|k|^{1/2}
$$
if $\cond(\sheaf{F})>\frac{4}{10}|k|^{1/10}$. In particular, if we
apply the last lemma with $\gamma=\frac{1}{10}$, any fixed $N\geq 1$,
and the corresponding constant $\alpha\geq 1$, we deduce that
$$
\proba\Bigl( \Bigl| \sum_{x\in k}\frtr{\sheaf{F}}{k}{x} \varphi(x)
\Bigr|\geq \alpha' \cond(\sheaf{F})^6\sqrt{|k|\log |k|}, \text{ for
  \emph{some} $\sheaf{F}$ in $\mext(k$)} \Bigr)\ll |k|^{-N}
$$
where $\alpha'=\max(\alpha,100)$. Then taking into
account~(\ref{eq-concentr-2}), we see that if we take
$$
s=6,\quad\quad \gamma=1/2,\quad\quad A=\alpha\sqrt{\log |k|},
$$
then the probability that $\varphi$ does not satisfy the conditions of
Proposition~\ref{pr-lower} for these values is $\ll
|k|^{-N}$. Therefore, we obtain our result using the
upper-bound~(\ref{eq-trivial-norm}).
\end{proof}

\end{document}